\documentclass[12pt,reqno]{amsart}
\usepackage{amsmath,amsthm,amscd,amsfonts,amssymb,color}
\usepackage{cite}
\usepackage[mathscr]{eucal}
\usepackage[bookmarksnumbered,colorlinks,plainpages]{hyperref}
\newtheorem{theorem}{Theorem}[section]

\newtheorem{open question}[theorem]{Open Question}
\theoremstyle{definition}
\newtheorem{definition}[theorem]{Definition}
\newtheorem{example}[theorem]{Example}
\newtheorem{remark}[theorem]{Remark}
\numberwithin{equation}{section}
\def\DJ{\leavevmode\setbox0=\hbox{D}\kern0pt\rlap
 {\kern.04em\raise.188\ht0\hbox{-}}D}

\begin{document}
\title[On Proximal contractions and Best proximity points ]{On Proximal contractions via implicit relations and Best proximity points }

\author[P.\ Mondal, H.\ Garai, L.K. \ Dey]
{Pratikshan Mondal$^{1}$, Hiranmoy Garai$^{2}$, Lakshmi Kanta Dey$^{3}$}

\address{{$^{1}$\,} Department of Mathematics,                                         					\newline \indent Durgapur Government College, 
                    Durgapur, India.}
                    \email{real.analysis77@gmail.com}
\address{{$^{2}$\,} Department of Mathematics,
                    \newline \indent National Institute of Technology
                    Durgapur, India.}
                    \email{hiran.garai24@gmail.com}
\address{{$^{3}$\,} Department of Mathematics,
                   \newline \indent National Institute of Technology
                    Durgapur, India.}
                    \email{lakshmikdey@yahoo.co.in}
\subjclass[2010]{$47$H$10$, $54$H$25$.}
\begin{abstract}
In this paper, we employ two types of implicit relations to define some new kind of proximal contractions and study about their best proximity points. More precisely, we use two class of functions $\mathcal{A}$ and $\mathcal{A}'$ to explore proximal $\mathcal{A}$, $\mathcal{A}'$-contractions of first and second type and strong proximal $\mathcal{A}$, $\mathcal{A}'$-contractions. 
We investigate the existence of best proximity point results of the same. It is worth mentioning that the well-known results of Sadiq Basha [J. Approx. Theory, $2011$] on proximal contractions are the special cases of our obtained results. We authenticate our results by suitable examples.

\vspace{0.5 cm}

\end{abstract}

\maketitle

\section{Introduction }\label{sec:1}
Best proximity point theory deals with a natural generalization of fixed point theory by routing the method of computing an optimal approximate solution to the equation $Sx=x$, where $S:G \to  H$ a non-self mapping, $G,H$ being two disjoint subsets of a metric space $(M,d)$. Since for $x\in G$, we always have $d(x,Sx)\geq dist(G,H)$, where   $dist(G, H)=\inf\{d(x, y): x\in G, y\in H\}$, it follows that an element $x\in G$ will be approximate optimal solution of $Sx=x$ if 
$d(x,Sx)= dist(G,H)$. Such a point $`x$' is known as best proximity point of $S$, and the branch of mathematics dealing with best proximity points is known as best proximity point theory. There are a numerous number of articles which analyze several kinds of contractions  for the existence of  best proximity point(s) for single-valued as well as multivalued mappings. Interested readers may consult with the papers \cite{B11, B19, AV09, KRR19, K12, MD17} for single-valued mappings and \cite {AS09, KKL08, KRV03, SV97, WPB09} for multivalued mappings.

The study of best proximity point theory by using different contractions had been enriched in $2011$ with a new kind of contraction by Sadiq Basha \cite{BS11}. In \cite{BS11}, he came with some new kind of contractions such as proximal contraction of the first kind, proximal contraction of the second kind, strong proximal contraction of the first kind.  
\begin{definition} { (\cite[p. 3, Definitions 2.2-2.4]{BS11}).}
Let $(M, d)$ be a metric space and  $G, H$  two non-empty subsets of $M$. A mapping $S:G\to H$ is said to be a 
\begin{itemize}
\item[{(i)}]
proximal contraction of the first kind if there exists $\alpha\in [0,1)$ satisfying 
$$ \left. \begin{array}{ll}
    d(u_1, Sx_1)=dist(G, H)\\
    d(u_2, Sx_2)=dist(G, H)
\end{array}\right\}\\ 
\Longrightarrow d(u_1, u_2)\le \alpha\ d(x_1, x_2)
$$  for all $u_1, u_2, x_1, x_2\in G$,
\item[{(ii)}] proximal contraction of the second kind if there exists an $\alpha\in [0,1)$ satisfying 
$$ \left. \begin{array}{ll}
    d(u_1, Sx_1)=dist(G, H)\\
    d(u_2, Sx_2)=dist(G, H)
\end{array}\right\}\\ 
\Longrightarrow d(Su_1, Su_2)\le \alpha\ d(Sx_1, Sx_2)
$$  for all $u_1, u_2, x_1, x_2\in G$,
\item [{(iii)}] a strong proximal contraction of the first kind if there exists $\alpha\in [0,1)$ such that for all $u_1, u_2, x_1, x_2\in G$ and for all $\gamma\in [1, 2)$
$$ \left. \begin{array}{ll}
    d(u_1, Sx_1)\le \gamma dist(G, H)\\
    d(u_2, Sx_2)\le \gamma dist(G, H)
\end{array}\right\}\\ 
\Rightarrow d(u_1, u_2)\le \alpha\ d(x_1, x_2)+(\gamma-1)dist(G, H).
$$
\end{itemize}
\end{definition}

In the above definitions of proximal contractions, we see that the definitions involves  the displacement $d(x_1,x_2)$ only. It is known that for two points $x_1,x_2$, the other displacements are $d(Sx_1,x_1),\ d(Sx_2,x_2),\ d(Sx_1,x_2)$ and $ d(Sx_2,x_1)$, and there are a plenty number of contractions which involves these displacements, and these contractions play a crucial role in the theory of fixed point and best proximity point.   If we compare Definition \ref{d1} with some usual well-known contractions, then one can notice that $u_1,u_2$ play the roles of $Sx_1, Sx_2$ in Definition \ref{d1}. So if someone requires to extend the proximal contractions by using the displacements $d(Sx_1,x_1),$ $d(Sx_2,x_2),\ d(Sx_1,x_2),\ d(Sx_2,x_1)$, then one has to work with $d(u_1,x_1)$, $(u_2,x_2)$, $d(u_1,x_2)$, $d(u_2,x_1)$ respectively. So it will be  impressive works if the concepts of proximal contractions can be enlarged by involving the displacements $d(u_1,x_1)$, $(u_2,x_2)$, $d(u_1,x_2)$, $d(u_2,x_1)$.

Motivated by this fact, in the current paper, we broaden the proximal contractions  by associating all the five displacements  $d(x_1,x_2)$, $d(u_1,x_1)$, $ d(u_2,x_2),$ $d(u_1,x_2)$ and $d(u_2,x_1)$. To continue this, we introduce proximal $\mathcal{A}$-contractions which involve  $d(x_1,x_2),$ $ d(u_1,x_1)$ and $ d(u_2,x_2)$; and proximal $\mathcal{A}'$-contractions which involve  $d(x_1,x_2),$ $ d(u_1,x_2)$ and $ d(u_2,x_1)$. More specifically, we define \textit{proximal $\mathcal{A}$-contractions of first and second type;  proximal $\mathcal{A}'$-contractions of first and second type;}  \textit{strong proximal $\mathcal{A}$-contraction} and \textit{strong proximal $\mathcal{A}'$-contraction}. After this, we study on adequate sufficient conditions to ensure the existence of best proximity point(s) of the above-mentioned contractions, and access the required adequate sufficient conditions which will be presented in next section. Along with this, we give a number of examples to support the validity of our proven results.

Throughout this paper, $\mathcal{A}$ and $\mathcal{A}'$ will contain all functions $f:\mathbb{R}_+^3\to \mathbb{R}$ having the  properties $(\mathcal{A}_1)$-$(\mathcal{A}_2)$ and $(\mathcal{A}'_1)$-$(\mathcal{A}'_3)$ respectively, where
\begin{itemize}
\item[{($\mathcal{A}_1$)}] there exists $k\in [0, 1)$ such that if $r\le f(s,s,r)$ or $r\le f(r, s, s)$, then  $r\le ks$ for all $r,s\in \mathbb{R}_+$;
\item[{($\mathcal{A}_2$)}] there exists $\alpha\in [0, 1)$ such that $f(r,0,0)\le \alpha r$;
\end{itemize}
and
\begin{itemize}
\item[{($\mathcal{A}'_1$)}] there exists $k\in [0, 1)$ such that if $r \le f(s,0,  r + s)$, then $r\le ks$ for all $r,s\in \mathbb{R}_+$;
\item[{($\mathcal{A}'_2$)}] if $t\le t_1$, then $f(r,s,t)\le f(r,s,t_1 )$ for all $r,s, t, t_1\in \mathbb{R}_+$;
\item[{($\mathcal{A}'_3$)}] if $r\le f(r, r, r)$, then $r=0$.
\end{itemize}
For examples and properties of such collection of mappings, we refer  the readers to \cite{AZS08, GDC, MGD}.
\section{Main results and their proofs}
Throughout this section, $(M,d)$ will denote a metric space and $G,\ H$ will denote two non-empty subsets of $M$, and $G_0$, $H_0$ will denote the following:
\begin{align*}
G_0&=\{x \in G : d(x, y) = dist(G,H)\ \mbox{ for some}\ y \in H\}\\
H_0&=\{y \in H : d(x, y) = dist(G,H)\ \mbox{ for some}\ x \in G\}.
\end{align*}

First, we define proximal $\mathcal{A}$, $\mathcal{A}'$ -contractions  of the first kind in the following way:
\begin{definition}  \label{d1}
 A mapping $S:G\to H$ is said to be a 
\begin{itemize}
\item[{(i)}]
proximal $\mathcal{A}$-contraction of the first type if there exists an $f\in \mathcal{A}$ satisfying 
$$ \left. \begin{array}{ll}
    d(u_1, Sx_1)=dist(G, H)\\
    d(u_2, Sx_2)=dist(G, H)
\end{array}\right\}\\ 
\Longrightarrow d(u_1, u_2)\le f(d(x_1, x_2), d(u_1, x_1), d(u_2, x_2))
$$  for all $u_1, u_2, x_1, x_2\in G$,

\item[{(ii)}]  proximal $\mathcal{A}'$-contraction of the first type if there exists an $f\in \mathcal{A}'$ satisfying 
$$ \left. \begin{array}{ll}
    d(u_1, Sx_1)=dist(G, H)\\
    d(u_2, Sx_2)=dist(G, H)
\end{array}\right\}\\ 
\Longrightarrow d(u_1, u_2)\le f(d(x_1, x_2), d(u_1, x_2), d(u_2, x_1))
$$ for all $u_1, u_2, x_1, x_2\in G$.
\end{itemize}
\end{definition}
Our first two results regarding the existence of best proximity point of the above two proximal contractions are as follows:
\begin{theorem}\label{t4}
Suppose that $(M,d)$ is complete, $G, H$ are   closed and $G_0\ne \emptyset.$
Let $S:G\to H$ be a continuous proximal $\mathcal{A}$-contraction of the first kind such that $S(G_0)$ resides in $H_0$. Then $S$ has a unique best proximity point.
\end{theorem}
\begin{proof}
Since $G_0$ is non-empty, we choose an element $u_0\in G$. Then $Su_0\in S(G_0)\subset H_0$. Then we find an element $u_1\in G_0$ such that $d(u_1, Su_0)=dist(G, H)$.

Similarly, $Su_1\in H_0$ and in the same way we find an element $u_2\in G_0$ such that $d(u_2, Su_1)=dist(G, H)$. 

Continuing this process, we arrive at a sequence $\{u_n\}$ of elements of $G_0$ such that 
$$d(u_{n+1}, Su_n)=dist(G, H)\ \mbox{ for all }n\in \mathbb{N}.$$
Now note that
$$d(u_n, Su_{n-1})=dist(G, H)$$
and
$$d(u_{n+1}, Su_n)=dist(G, H)$$
for all $n\in \mathbb{N}$.

Since $S$ is a proximal $\mathcal{A}$-contraction of the first type, there exists an $f\in \mathcal{A}$ such that
$$d(u_n, u_{n+1})\le f(d(u_{n-1}, u_n), d(u_n, u_{n-1}), d(u_{n+1}, u_n)).$$
So there exists a $k\in [0, 1)$ such that
$$d(u_n, u_{n+1})\le kd(u_{n-1}, u_{n})$$
for all $n\in \mathbb{N}$ which, in fact, implies that
$$d(u_n, u_{n+1})\le k^nd(u_1, u_0).$$
Now for any $m,n\in \mathbb{N}$, we have
\begin{align*}
d(u_{m+n}, u_n)&\le d(u_{m+n}, u_{m+n-1})+d(u_{m+n-1}, u_{m+n-2})+\cdots +d(u_{n+1}, u_n)\\
&\le (k^{m+n-1}+k^{m+n-2}+\cdots+k^n)d(u_1, u_0)\\
&=k^n\frac{1-k^m}{1-k}d(u_1, u_0)\longrightarrow 0 \mbox{ as }m,n\to \infty.
\end{align*}
Therefore, $\{u_n\}$ is a Cauchy sequence in $G$. Being a closed subset of a complete metric space $(M, d)$, $G$ supply an element $u$ such that $u_n\longrightarrow u$ as $n\to \infty$.

Then, by continuity of $S$, we get $Su_n\to Su$ as $n\to \infty$ and consequently $d(u_{n+1}, Su_n)\to d(u, Su)$.

Now $d(u_{n+1}, Su_n)=dist(G, H)$ for all $n\in \mathbb{N}$, confirms that $d(u, Su)=d(G, H)$ which shows that $u$ is a best proximity point of $S$.

Let $u^*\in G$ be such that $d(u^*, Su^*)=dist(G, H)$. Then we have
$$d(u, u^*)\le f((d(u, u^*), d(u, u), d(u^*, u^*))=f(d(u, u^*), 0, 0)$$
which implies that
$$d(u, u^*)\le k\cdot 0=0.$$ 
Hence $u=u^*$ and the theorem is proved.
\end{proof}
\begin{theorem}\label{t8}
Suppose that $(M,d)$ is complete, $G, H$ are   closed and $G_0\ne \emptyset.$
Let $S:G\to H$ be a continuous proximal $\mathcal{A}'$-contraction of the first kind such that $S(G_0)$ resides in $H_0$. Then $S$ has a unique best proximity point in $G$.
\end{theorem}
\begin{proof}
We consider a sequence $\{u_n\}$ of elements of $G_0$, defined as in Theorem~\ref{t4}, such that 
$$d(u_{n+1}, Su_n)=dist(G, H)\ \mbox{ for all }n\in \mathbb{N}.$$
Now note that
$$d(u_n, Su_{n-1})=dist(G, H)$$
and
$$d(u_{n+1}, Su_n)=dist(G, H)$$
for all $n\in \mathbb{N}$.

Since $S$ is a proximal $\mathcal{A}'$-contraction of the first type, there exists an $f\in \mathcal{A}'$ such that
\begin{align*}
d(u_n, u_{n+1})&\le f(d(u_{n-1}, u_n), d(u_n, u_{n}), d(u_{n+1}, u_{n-1}))\\
&\le f(d(u_{n-1}, u_n), 0, d(u_{n+1}, u_{n})+d(u_n, u_{n-1})).
\end{align*}
So there exists a $k\in [0, 1)$ such that
$$d(u_n, u_{n+1})\le kd(u_{n-1}, u_{n})$$
for all $n\in \mathbb{N}$ which, in fact, implies that
$$d(u_n, u_{n+1})\le k^nd(u_1, u_0).$$
Now for any $m,n\in \mathbb{N}$, we have
\begin{align*}
d(u_{m+n}, u_n)&\le d(u_{m+n}, u_{m+n-1})+d(u_{m+n-1}, u_{m+n-2})+\cdots +d(u_{n+1}, u_n)\\
&\le (k^{m+n-1}+k^{m+n-2}+\cdots+k^n)d(u_1, u_0)\\
&=k^n\frac{1-k^m}{1-k}d(u_1, u_0)\longrightarrow 0 \mbox{ as }m,n\to \infty.
\end{align*}
Therefore $\{u_n\}$ is a Cauchy sequence in $G$ an since $G$ is a closed subset of the complete metric space $(M, d)$, $u_n\longrightarrow u$ as $n\to \infty$ for some $u\in G$.

Applying continuity of $S$, we find that $Su_n\to Su$ as $n\to \infty$ and therefore $d(u_{n+1}, Su_n)\to d(u, Su)$.

As $d(u_{n+1}, Su_n)=dist(G, H)$ for all $n\in \mathbb{N}$, we have $d(u, Su)=dist(G, H)$ i.e., $u$ is a best proximity point of $S$.

Let $u^*\in G$ be such that $d(u^*, Su^*)=dist(G, H)$. Since $S$ is a proximal $\mathcal{A}'$-contraction of the first type, we have
$$d(u, u^*)\le f(d(u, u^*), d(u^*, u), d(u, u^*)$$
which implies that
$$d(u, u^*)=0.$$ 
Hence $u=u^*$ and the proof is complete.
\end{proof}
Next, we give the following supporting examples:
\begin{example}
We take $M=\mathbb{R}$, $d$ as  the usual metric and choose $G=[2, \infty)$, $H=(-\infty, -1]$. Also we take $f\in \mathcal{A}$ defined by $f(r, s, t)=\frac{3}{4} \max\{r, s, t\}$ and define $S:G\to H$  by $Sx=\frac{2-3x}{4}$ for all $x\in G$.


Let $u_1, u_2, x_1, x_2\in G$ be such that $d(u_1, Sx_1)=dist(G, H)$ and $d(u_2, Sx_2)=dist(G, H)$. Then
$$4u_1+3x_1=14\ \mbox{ and }\ 4u_2+3x_2=14.$$
Now,
\begin{align*}
d(u_1, u_2)&=|u_1-u_2|\\
&=\left|\frac{14-3x_1}{4}-\frac{14-3x_2}{4}\right|\\
&=\frac{3}{4}|x_1-x_2|=\frac{3}{4}d(x_1, x_2),
\end{align*}
which yields that
$$d(u_1, u_2)\le \frac{3}{4}f(d(x_1, x_2), d(u_1, x_1), d(u_2, x_2)).$$
Therefore, $S$ is a proximal $\mathcal{A}$-contraction of first type.
So by Theorem \ref{t4}, $S$ has a unique best proximity point, viz., $u=2$.
\end{example}

\begin{example}
We choose $M=\mathbb{R}$, $d$ as the usual metric; $G=[6, 7]$, $H=[2, 3]$; $f(r, s, t)=\frac{49}{50} \max\{s, t\}$ and define
$S:G\to H$ be defined by $Sx=9-x$ for all $x\in G$.

Let $u_1, u_2, x_1, x_2\in G$ be such that $d(u_1, Sx_1)=dist(G, H)$ and $d(u_2, Sx_2)=dist(G, H)$.

Then
$$u_1+x_1=12\ \mbox{ and }\ u_2+x_2=12.$$
Without loss of generality, let us suppose that $x_1\ge x_2$. Then
\begin{align*}
d(u_1, u_2)&=|u_1-u_2|\\
&=|12-x_1-12+x_2|\\
&=x_1-x_2.
\end{align*}
Also, 
$$d(u_1, x_1)=12-2x_1\ \mbox{ and }\ d(u_2, x_2)=12-2x_2.$$
Now,
\begin{align*}
&f(d(x_1, x_2), d(u_1, x_1), d(u_2, x_2))\\
&=\frac{49}{50}\max\{x_1-x_2, 2x_1-12, 2x_2-12\}\\
&=\frac{49}{50}(2x_1-12)\ \Big[\because x_1\ge x_2, \mbox{ so, } 2x_1-12\ge 2x_2-12\Big].
\end{align*}
Therefore,
$$d(u_1, u_2)\le f(d(x_1, x_2), d(u_1, x_1), d(u_2, x_2))$$
which shows that $S$ is a proximal $\mathcal{A}$-contraction of first type. So by Theorem \ref{t4}, $S$ possesses a unique best proximity point, viz., $u=6$.
\end{example}
\begin{example}
We choose $(M,d)$ as the usual metric space $(\mathbb{R}, d)$ and $G=[3, 5]$, $H=[0, 1]$. We take $f\in \mathcal{A}'$ as $f(r, s, t)=\frac{1}{3}(s+t)$ and consider the mapping $S:G\to H$ defined by
$$ Sx=\left\{%
\begin{array}{ll}
    1 &\hbox{if $x\in [3, 4]$};\\
    5-x & \hbox{if $x\in [4, 5]$}.
\end{array}%
\right.$$ 
Let $u_1, u_2, x_1, x_2\in G$ be such that $d(u_1, Sx_1)=dist(G, H)=d(u_2, Sx_2)$.

We now consider the following cases:

\textbf{Case 1:} Let $x_1, x_2\in [3, 4]$. Then
$$|u_1-1|=2\Longrightarrow u_1=3.$$
Similarly, $u_2=3$.

So, it is obvious that
$$d(u_1, u_2)\le f\Big(d(x_1, x_2), d(u_1, x_2), d(u_2, x_1)\Big).$$

\textbf{Case 2:} Let $x_1, x_2\in [4, 5]$. Then
\begin{align*}
&|u_1-(5-x_1)|=2\\
&\Longrightarrow |u_1+x_1-5|=2\\
&\Longrightarrow u_1+x_1=7.
\end{align*}
Similarly, $u_2+x_2=7$.

Therefore, $d(u_1, u_2)=|u_1-u_2|=|x_1-x_2|$. Without loss of generality, we assume that $x_1\ge x_2$.

Again,
\begin{align*}
d(u_1, x_2)&=|u_1-x_2|\\
&=|7-x_1-x_2|\\
&=x_1+x_2-7.
\end{align*}
Similarly, $d(u_2, x_1)=x_1+x_2-7$.

Therefore,
\begin{align*}
&3d(u_1, u_2)-\{d(u_1, x_2)+d(u_2, x_1)\}\\
&=3(x_1-x_2)-\{x_1+x_2-7+x_1+x_2-7\}\\
&=3x_1-3x_2-2x_1-2x_2+14\\
&=x_1-5x_2+14\\
&\le 5-20+14=-1<0
\end{align*}
which gives
$$d(u_1, u_2)\le \frac{1}{3}\{d(u_1, x_2)+d(u_2,x_1)\}$$
that is
$$d(u_1, u_2)\le f\Big(d(x_1, x_2), d(u_1, x_2), d(u_2,x_1)\Big).$$

\textbf{Case 3:} Let $x_1\in [3, 4]$ and $x_2\in [4, 5]$. Then as in the above cases, we have $u_1=3$ and $u_2+x_2=7$.

Therefore,
\begin{align*}
d(u_1, u_2)&=|u_1-u_2|\\
&=|3-u_2|=u_2-3\\
&=4-x_2.
\end{align*}
Now,
\begin{align*}
d(u_1, x_2)+d(u_2, x_1)&=|3-x_2|+|u_2-x_1|\\
&=x_2-3+x_1+x_2-7\\
&=x_1+2x_2-10.
\end{align*}
Therefore, as in case-2, it can be shown that 
$$d(u_1, u_2)\le f\Big(d(x_1, x_2), d(u_1, x_2), d(u_2,x_1)\Big).$$
Hence combining all the cases, we see that $S$ is a proximal $\mathcal{A}'$-contraction of first type. Hence Theorem \ref{t8} ensures that $S$ admits a unique best proximity point. Note that the best proximity point is $3$.

\end{example}
Next, we give the definitions of proximal $\mathcal{A}$, $\mathcal{A}'$- contractions of the second type.
\begin{definition} \label{d2}
 A mapping $S:G\to H$ is said to be a 
\begin{itemize}
\item[{(i)}] proximal $\mathcal{A}$-contraction of the second type if there exists an $f\in \mathcal{A}$ satisfying 
$$ \left. \begin{array}{ll}
    d(u_1, Sx_1)=dist(G, H)\\
    d(u_2, Sx_2)=dist(G, H)
\end{array}\right\}\\ 
$$
$$\Longrightarrow d(Su_1, Su_2)\le f(d(Sx_1, Sx_2), d(Su_1, Sx_1), d(Su_2, Sx_2))
$$  for all $u_1, u_2, x_1, x_2\in G$,
\item[{(ii)}] proximal $\mathcal{A}'$-contraction of the second type if there exists an $f\in \mathcal{A}'$ satisfying 
$$ \left. \begin{array}{ll}
    d(u_1, Sx_1)=dist(G, H)\\
    d(u_2, Sx_2)=dist(G, H)
\end{array}\right\}\\ 
$$
$$\Longrightarrow d(Su_1, Su_2)\le f(d(Sx_1, Sx_2), d(Su_1, Sx_2), d(Su_2, Sx_1))
$$ for all $u_1, u_2, x_1, x_2\in G$.
\end{itemize}
\end{definition}
Our upcoming two results deal with the existence of best proximity point of the aforementioned contractions.
Before presenting these results, we first recall the following definition:
\begin{definition} { (\cite[p. 3, Definition 2.1]{BS11}).}
 $G$ is said to be approximatively compact with respect to $H$ if  every sequence $\{x_n\}$ in $G$ with $d(y, x_n) \to d(y, G)$ for some $y$ in $H$, has a convergent subsequence.
\end{definition}

\begin{theorem}\label{t5}
Suppose that $(M,d)$ is complete, $G, H$ are   closed, $G$ is approximately compact with respect to $H$ and $G_0\ne \emptyset.$
Let $S:G\to H$ be a continuous proximal $\mathcal{A}$-contraction of the second type such that $S(G_0)$ resides in $H_0$. Then $S$ has a  best proximity point in $G$.  Moreover, if $S$ is injective, then the best proximity point is unique.
\end{theorem}
\begin{proof}
Since $G_0$ is non-empty, we choose an element $v_0\in G$. Then $Sv_0\in S(G_0)\subset H_0$. Then there is an element $v_1\in G_0$ such that $d(v_1, Sv_0)=dist(G, H)$.

Similarly, $Sv_1\in H_0$ and in the same way we find an element $v_2\in G_0$ such that $d(v_2, Sv_1)=dist(G, H)$. 

Therefore, continuing this process we arrive at a sequence $\{v_n\}$ of elements of $G_0$ such that 
$$d(v_{n+1}, Sv_n)=dist(G, H)\ \mbox{ for all }n\in \mathbb{N}.$$
Now note that
$$d(v_n, Sv_{n-1})=dist(G, H)$$
and
$$d(v_{n+1}, Sv_n)=dist(G, H)$$
for all $n\in \mathbb{N}$.

Since $S$ is a proximal $\mathcal{A}$-contraction of the second type, there exists an $f\in \mathcal{A}$ such that
$$d(Sv_n, Sv_{n+1})\le f(d(Sv_{n-1}, Sv_n), d(Sv_n, Sv_{n-1}), d(Sv_{n+1}, Sv_n)).$$
So there exists a $k\in [0, 1)$ such that
$$d(Sv_n, Sv_{n+1})\le kd(Sv_{n-1}, Sv_{n})$$
for all $n\in \mathbb{N}$ which, in fact, implies that
$$d(Sv_n, Sv_{n+1})\le k^nd(Sv_1, Sv_0).$$
Now for any $m,n\in \mathbb{N}$, we have
\begin{align*}
d(Sv_{m+n}, Sv_n)&\le d(Sv_{m+n}, Sv_{m+n-1})+d(Sv_{m+n-1}, Sv_{m+n-2})+\cdots +d(Sv_{n+1}, Sv_n)\\
&\le (k^{m+n-1}+k^{m+n-2}+\cdots+k^n)d(Sv_1, Sv_0)\\
&=k^n\frac{1-k^m}{1-k}d(Sv_1, Sv_0)\longrightarrow 0 \mbox{ as }m,n\to \infty.
\end{align*}
This shows that $\{Sv_n\}$ is a Cauchy sequence in $H$. Now closedness of $H$ in the complete metric space $(M, d)$ ensures the existence of an element $v\in H$ such that $Sv_n\longrightarrow v$ as $n\to \infty$.

Now,
\begin{align*}
dist(v, G)&\le d(v, v_n)\\
&\le d(v, Sv_{n-1})+d(Sv_{n-1}, v_n)\\
&=d(v, Sv_{n-1})+dist(G, H)\\
&\le d(v, Sv_{n-1})+dist(v, G)
\end{align*}
which implies that $d(v, v_n)\to dist(v, G)$ as $n\to \infty$.

Since $G$ is proximally compact with respect to $H$, $\{v_n\}$ has a convergent subsequence $\{v_{n_k}\}$ in $G$. Let $v_{n_k}\to u$ for some $u\in G$. 

Then
$$d(u, v)=\lim_{k\to \infty}d(v_{n_k}, Sv_{n_k-1})=dist(G, H).$$
Therefore, $u\in G_0$. 

Since $S$ is continuous, $Sv_{n_k}\to Su$ as $k\to \infty$. Again we have, $Sv_{n_k}\to v$ as $k\to \infty$. Hence $v=Su$.

Thus, $d(u, Su)=dist(G, H)$.

Finally, let $S$ be injective. Let $u^*$ be another element in $G$ such that $d(u^*, Su^*)=dist(G, H)$. Then,
$$d(Su, Su^*)\le f(d(Su, Su^*), d(Su, Su), d(Su^*, Su^*))=f(d(Su, Su^*), 0, 0)$$ which implies that
$$d(Su, Su^*)\le k\cdot 0=0.$$
Hence $Su=Su^*$. Since $S$ is injective, we have $u=u^*$ and the proof is complete.
\end{proof}
\begin{theorem}\label{t9}
Suppose that $(M,d)$ is complete, $G, H$ are   closed, $G$ is approximately compact with respect to $H$ and $G_0\ne \emptyset.$
Let $S:G\to H$ be a continuous proximal $\mathcal{A}'$-contraction of the second type such that $S(G_0)$ resides in $H_0$. Then $S$ has a  best proximity point in $G$.  Moreover, if $S$ is injective, then the best proximity point is unique.
\end{theorem}
\begin{proof}
Proceeding as in Theorem~\ref{t5}, we can construct a sequence $\{v_n\}$ of elements of $G_0$ such that 
$$d(v_{n+1}, Sv_n)=dist(G, H)\ \mbox{ for all }n\in \mathbb{N}.$$
Now note that
$$d(v_n, Sv_{n-1})=dist(G, H)$$
and
$$d(v_{n+1}, Sv_n)=dist(G, H)$$
for all $n\in \mathbb{N}$.

Since $S$ is a proximal $\mathcal{A}'$-contraction of the second type, there exists an $f\in \mathcal{A}'$ such that
\begin{align*}
d(Sv_n, Sv_{n+1})&\le f(d(Sv_{n-1}, Sv_n), d(Sv_n, Sv_{n}), d(Sv_{n+1}, Sv_{n-1}))\\
&\le f(d(Sv_{n-1}, Sv_n), 0, d(Sv_{n+1}, Sv_{n})+d(Sv_{n}, Sv_{n-1})).
\end{align*}
So there exists a $k\in [0, 1)$ such that
$$d(Sv_n, Sv_{n+1})\le kd(Sv_{n-1}, Sv_{n})$$
for all $n\in \mathbb{N}$. Repeated use of the above, we get
$$d(Sv_n, Sv_{n+1})\le k^nd(Sv_1, Sv_0).$$
Now for any $m,n\in \mathbb{N}$, we have
\begin{align*}
d(Sv_{m+n}, Sv_n)&\le d(Sv_{m+n}, Sv_{m+n-1})+d(Sv_{m+n-1}, Sv_{m+n-2})+\cdots +d(Sv_{n+1}, Sv_n)\\
&\le (k^{m+n-1}+k^{m+n-2}+\cdots+k^n)d(Sv_1, Sv_0)\\
&=k^n\frac{1-k^m}{1-k}d(Sv_1, Sv_0)\longrightarrow 0 \mbox{ as }m,n\to \infty.
\end{align*}
Therefore $\{Sv_n\}$ is a Cauchy sequence in $H$. Since $H$ is a closed subset of the complete metric space $(M, d)$, we get an element $v\in H$ such that $Sv_n\longrightarrow v$ as $n\to \infty$.

Now,
\begin{align*}
dist(v, G)&\le d(v, v_n)\\
&\le d(v, Sv_{n-1})+d(Sv_{n-1}, v_n)\\
&=d(v, Sv_{n-1})+dist(G, H)\\
&\le d(v, Sv_{n-1})+dist(v, G)
\end{align*}
which implies that $d(v, v_n)\to dist(v, G)$ as $n\to \infty$.

Since $G$ is proximally compact with respect to $H$, $\{v_n\}$ has a convergent subsequence $\{v_{n_k}\}$ in $G$. Let $v_{n_k}\to u$ for some $u\in G$. 

Now
$$d(u, v)=\lim_{k\to \infty}d(v_{n_k}, Sv_{n_k-1})=dist(G, H).$$
This implies that $u\in G_0$. 

Applying continuity of $S$, we get $Sv_{n_k}\to Su$ as $k\to \infty$. Again we have, $Sv_{n_k}\to v$ as $k\to \infty$. Hence $v=Su$.

Thus, $d(u, Su)=dist(G, H)$.

We now take $S$ to be injective. Let $u^*$ be another element in $G$ such that $d(u^*, Su^*)=dist(G, H)$. Then,
$$d(Su, Su^*)\le f(d(Su, Su^*), d(Su, Su^*), d(Su^*, Su))$$ which implies that
$$d(Su, Su^*)=0.$$
Hence $Su=Su^*$. Since $S$ is injective, we have $u=u^*$ and the proof is complete.
\end{proof}
\begin{remark}
In the above two theorems, to ensure the uniqueness of best proximity point, injectiveness of $S$ is not necessary, which follows from the following examples.
\end{remark}
\begin{example}
We take $(M,d)=(\mathbb{R}^2,d)$ where
$d\Big((x_1, y_1), (x_2, y_2)\Big)=|x_1-x_2|+|y_1-y_2|$
for all $(x_1, y_1), (x_2, y_2)\in \mathbb{R}^2$;  $G=\{(x, y)\in \mathbb{R}^2:4\le x\le 5, 0\le y\le 1\}$, $H=\{(x, y)\in \mathbb{R}^2:0\le x\le 1, 0\le y\le 1\}$; $f(r, s, t)=\frac{1}{2}r+\frac{1}{5}(s+t)$, and 
 define $S:G\to H$ by
$$S(x, y)=\left(1, \frac{y}{2}\right)$$
for all $(x, y)\in \mathbb{R}^2$.

Let $u_1=(u'_1, u''_1), u_2=(u'_2, u''_2), x_1=(x'_1, x''_1), x_2=(x'_2, x''_2)\in G$ be such that
$$d(u_1, Sx_1)=dist(G, H)=3 \mbox{ and } d(u_2, Sx_2)=dist(G, H)=3.$$

Then, $Sx_1=\left(1, \frac{x''_1}{2}\right)$ and $Sx_2=\left(1, \frac{x''_2}{2}\right)$. Now,
\begin{align*}
&d(u_1, Sx_1)=3\\
&\Longrightarrow d\left((u'_1, u''_1), \left(1, \frac{x''_1}{2}\right)\right)=3\\
&\Longrightarrow |u'_1-1|+\left|u''_1-\frac{x''_1}{2}\right|=3\\
&\Longrightarrow u'_1-1+\left|u''_1-\frac{x''_1}{2}\right|=3\\
&\Longrightarrow u'_1+\left|u''_1-\frac{x''_1}{2}\right|=4
\end{align*}
which implies that $u'_1=4$ and $u''_1=\frac{x''_1}{2}.$
Similarly, $d(u_2, Sx_2)=3$ gives $u'_2=4$ and $u''_2=\frac{x''_2}{2}.$

Therefore,
\begin{align*}
d(Su_1, Su_2)&=d\left(\left(1, \frac{u''_1}{2}\right), \left(1, \frac{u''_2)}{2}\right)\right)\\
&=\left|\frac{u''_1}{2}-\frac{u''_2}{2}\right|\\
&=\left|\frac{x''_1}{4}-\frac{x''_2}{4}\right|=\frac{1}{4}|x''_1-x''_2|.
\end{align*}
Also, 
$$d(Sx_1, Sx_2)=d\left(\left(1, \frac{x''_1}{2}\right), \left(1, \frac{x''_2)}{2}\right)\right)=\frac{1}{2}|x''_1-x''_2|.$$
Now,
\begin{align*}
&d(Su_1, Su_2)-f(d(Sx_1, Sx_2), d(Su_1, Sx_1), d(Su_2, Sx_2))\\
&=d(Su_1, Su_2)-\left\{\frac{1}{2}d(Sx_1, Sx_2)+\frac{1}{5}\Big(d(Su_1, Sx_1) + d(Su_2, Sx_2)\Big)\right\}\\
&=\frac{1}{4}|x''_1-x''_2|-\frac{1}{4}|x''_1-x''_2|-\frac{1}{5}\Big(d(Su_1, Sx_1) + d(Su_2, Sx_2)\Big)\le 0
\end{align*}
which yields that
$$d(Su_1, Su_2)\le f(d(Sx_1, Sx_2), d(Su_1, Sx_1), d(Su_2, Sx_2))$$
which, in turn, implies that $S$ is a proximal $\mathcal{A}$-contraction of the second type.
 It is easy to check that $(4, 0)$ is the unique best proximity point of $S$ and $S$ is not injective.
\end{example}
\begin{example}
In this example, we take the metric space $(M,d)$ as above and choose
$$G=\Big\{(x, y)\in \mathbb{R}^2 : x=2, 0\le y\le 3\Big\}\bigcup \Big\{(x, y)\in \mathbb{R}^2 : 0\le x\le 2, y=2\Big\}$$
and
$$H=\Big\{(x, y)\in \mathbb{R}^2 : 0\le x\le 1, 0\le y\le 1\Big\}.$$
Then $dist(G, H)=1$. We define $S:G\to H$ by
$$S(x, y)=\left(\frac{x}{2}, 0\right)$$
for all $(x, y)\in \mathbb{R}^2$.
Also we choose $f\in \mathcal{A}'$ which is defined by $f(r,s,t)=\frac{1}{4}(s+t)$.
Let $u_1, u_2, x_1, x_2\in G$ be such that $d(u_1, Sx_1)=dist(G, H)=d(u_2, Sx_2)$. Then we have
\begin{align*}
&d\left((u'_1, u''_1), \left(\frac{x'_1}{2}, 0\right)\right)=1\\
&\Longrightarrow \left|u'_1-\frac{x'_1}{2}\right|+|u''_1|=1
\end{align*} 
which implies that $u''_1\le 1$ and $u'_1=2$.

Similarly, we get $u''_2\le 1$ and $u'_2=2$.

Now,
$$d(Su_1, Su_2)=d\left(\left(\frac{u'_1}{2},0\right), \left(\frac{u'_2}{2}, 0\right)\right)=\left|\frac{u'_1}{2}-\frac{u'_2}{2}\right|=0.$$

Therefore,
$$d(Su_1, Su_2)\le f(d(Sx_1, Sx_2), d(Sx_1, u_2), d(Sx_2, u_1))$$
whence $S$ is a proximal $\mathcal{A}'$-contraction of second type. 

One can easily  verify that $(4, 0)$ is the unique best proximity point of $S$ and $S$ is not injective.
\end{example}
\begin{remark}
In Theorem~\ref{t5} and Theorem~\ref{t9}, the injectiveness of $S$ can't be dropped, which follows from the following example.
\end{remark}
\begin{example}
Let us take $M=\mathbb{R}$, $d$ as the usual metric and $G=\left[ -1,-\frac{1}{2} \right] \cup \left[ \frac{1}{2} ,-1\right]$, $H=\{0\}$. We define $S:G\to H$ by $Sx=0$ for all $x\in G$. Then one can check that $S$ is proximal $\mathcal{A}$, $\mathcal{A}'$-contractions of the second type and $S$ has two best proximity points. It may be noted that $S$ is not an injection.
\end{example}
Next, we come up with the notions of strong proximal contractions, and present two results exhibiting the sufficient conditions in order to get best proximity points of strong proximal contractions.
\begin{definition} \label{d6}
A mapping $S:G\to H$ is said to be 
\begin{itemize}
\item[{(i)}] a strong proximal $\mathcal{A}$-contraction if there exists an $f\in \mathcal{A}$ such that for all $u_1, u_2, x_1, x_2\in G$ and for all $\gamma\in [1, 2)$
$$ \left. \begin{array}{ll}
    d(u_1, Sx_1)\le \gamma dist(G, H)\\
    d(u_2, Sx_2)\le \gamma dist(G, H)
\end{array}\right\}\\ 
$$
$$\Longrightarrow d(u_1, u_2)\le f(d(x_1, x_2), d(u_1, x_1), d(u_2, x_2))  +(\gamma-1)dist(G, H), 
$$
\item[{(ii)}] a strong proximal $\mathcal{A}'$-contraction if there exists an $f\in \mathcal{A}'$ such that for all $u_1, u_2, x_1, x_2\in G$ and for all $\gamma\in [1, 2)$
$$ \left. \begin{array}{ll}
    d(u_1, Sx_1)\le \gamma dist(G, H)\\
    d(u_2, Sx_2)\le \gamma dist(G, H)
\end{array}\right\}\\ 
$$
$$\Longrightarrow d(u_1, u_2)\le f(d(x_1, x_2), d(u_1, x_2), d(u_2, x_1))+(\gamma-1)dist(G, H).
$$
\end{itemize}
\end{definition}
\begin{theorem}\label{t7}
Suppose that $(M,d)$ is complete, $G,H$ are closed and $dist(G, H)>0$. Let $S:G\to H$ be a continuous strong proximal $\mathcal{A}$-contraction such that there exists a sequence $\{x_n\}$ in $G$ with $d(x_n, Sx_n)\to dist(G, H)$ as $n\to \infty$. Then $S$ has a unique best proximity point and     $\{x_n\}$ has a subsequence converging to that best proximity point.
\end{theorem}
\begin{proof}
For each $p\in \mathbb{N}$, we define
$$F_p=\left\{x\in G:d(x, Sx)\le \left(1+\frac{1}{p}\right)dist(G, H)\right\}.$$
Since $d(x_n, Sx_n)\to dist(G, H)$, there exists an $n_p\in \mathbb{N}$ such that
$$d(x_{n_p}, Sx_{n_p})\le \left(1+\frac{1}{p}\right)dist(G, H)$$
which implies that $F_p$ is non-empty for each $p\in \mathbb{N}$. Since $S$ is continuous, each $F_p$ is closed. It is also evident that $F_{p+1}\subset F_p$ for each $p\in \mathbb{N}$.

If $x$ and $x^*$ are two elements of $F_p$, then we have
$$d(x, Sx)\le \left(1+\frac{1}{p}\right)dist(G, H)$$ and
$$d(x^*, Sx^*)\le \left(1+\frac{1}{p}\right)dist(G, H).$$
Since $S$ is a strong proximal $\mathcal{A}$-contraction, there exists $f\in \mathcal{A}$ such that
\begin{align*}
d(x, x^*)&\le f(d(x, x^*), d(x, x), d(x^*, x^*))+\frac{1}{p}dist(G, H)\\
&= f(d(x, x^*), 0, 0)+\frac{1}{p}dist(G, H)\\
&\le \alpha d(x, x^*)+\frac{1}{p}dist(G, H) \mbox{ for some }\alpha\in [0, 1).
\end{align*}
Therefore, we get
$$d(x, x^*)\le \frac{1}{(1-\alpha)p}dist(G, H).$$
Hence $diam(A_p)\to 0$ as $p\to \infty$. Therefore by Cantor's intersection theorem, we have 
$$\bigcap_p F_p=\{u\}$$
for some $u\in G$.

From this we see that,
$$dist(G, H)\le d(u, Su)\le \left(1+\frac{1}{p}\right)dist(G, H)$$
for each $p$. Hence we have
$$d(u, Su)=dist(G, H).$$
For the last part, it is to be noted that
$$d(x_{n_p}, u)\le \frac{1}{(1-\alpha)p}dist(G, H).$$
Hence the subsequence $\{x_{n_p}\}$ converges to $u$ and the theorem follows.
\end{proof}
\begin{remark}
The conclusions of the above theorem also hold if $S$ is a strong proximal $\mathcal{A}'$-contraction  instead of strong proximal $\mathcal{A}$-contraction. The proof being similar to above theorem, we omit it.
\end{remark}
We conclude this paper by presenting an example in support of Theorem~\ref{t7}  followed by a couple of remarks.
\begin{example}
Let us take $(M,d)=(\mathbb{R}, d)$, $d$ being the usual metric; $G=[0, 1]$ and $H=[5, 6]$ and take $f\in \mathcal{A}$, where $f(r,s,t)=\frac{1}{4}(r+s+t)$.
We define $S:G\to H$ by $Sx=6-x$ for all $x\in G$.

Let $u_1, u_2, x_1, x_2\in G$ be such that $d(u_1, Sx_1)\le \gamma dist(G, H)$ and $d(u_2, Sx_2)\le \gamma dist(G, H)$ for all $\gamma\in [1, 2]$. Then
\begin{align*}
|u_1-Sx_1|&\le 4\gamma\\
\Longrightarrow |u_1-6+x_1|&\le 4\gamma\\
\Longrightarrow 6-u_1-x_1&\le 4\gamma\\
\Longrightarrow u_1&\ge 6-4\gamma-x_1.
\end{align*}
Similarly, 
$$u_2\ge 6-4\gamma-x_2.$$
Without loss of generality, we assume that $u_1\ge u_2$.
Therefore,
\begin{align*}
d(u_1, u_2)&=|u_1-u_2|\\
&=u_1-u_2\\
&\le u_1-(6-4\gamma-x_2)\\
&=u_1-6+4\gamma+x_2\\
&\le 1+1-6+4\gamma\\
&=4(\gamma-1)=(\gamma-1)d(A, B).
\end{align*}
So, we get
$$d(u_1, u_2)\le f(d(x_1, x_2), d(u_1, x_1), d(u_2, x_2))+(\gamma-1)dist(G, H)$$
for any $f\in \mathcal{A}$ which implies that $S$ is a strong proximal $\mathcal{A}$-contraction. Consequently by Theorem \ref{t7}, $S$ has a unique best proximity point. Also, the unique best proximity point is $1$.
\end{example}
\begin{remark}
The best proximity point results of different kind proximal contractions due to Sadiq Basha \cite{BS11} can be obtained from our results by choosing $f(r,s,t)=\alpha r$, where $\alpha\in [0,1)$.
\end{remark}
\begin{remark}
By selecting different $f$ in Theorem~\ref{t4}, Theorem~\ref{t5} and Theorem~\ref{t7}, we can obtain the best proximity point results of the proximal versions of the contractions of Kannan \cite {Rk68} $\big(f(r,s,t)=\alpha(s+t), ~~\mbox{where}~~ 0\leq \alpha<\frac{1}{2} \big)$, Reich \cite{R3} $\big(f(r,s,t)=\alpha_1r+\alpha_2s+\alpha_3t,$ $\mbox{where}~~0\leq \alpha_1,\alpha_2,\alpha_3<1; \alpha_1+\alpha_2+\alpha_3<1\big)$, Bianchini \cite{B10} $\big( f(r,s,t)=\alpha \max\{s,t\}$, where $0\leq \alpha<1\big)$ and Khan \cite{K78} $\big(f(r,s,t)=\alpha \sqrt{st}$, where $0\leq \alpha<1 \big)$.
\end{remark}

\paragraph{\textbf{Acknowledgement}.}
The second named author is funded by CSIR, New Delhi, INDIA (Award Number: $09/973(0018)/2017$-EMR-I). 

\bibliographystyle{plain}

\end{document}